\documentclass[10pt,reqno]{amsart}
\usepackage{hyperref,amsmath,amsfonts,amssymb,amsthm,enumerate}
\usepackage[mathcal]{euscript}
\usepackage[numbers]{natbib}
\usepackage[a5paper, left=8mm, right=8mm, top=1.5cm, bottom =1cm]{geometry}
\def\N{\mathbb{N}}
\def\d{{\fam0 d}}
\def\codim{\operatorname{codim}}
\def\dim{\operatorname{dim}}
\def\spn{\operatorname{span}}
\def\rank{\operatorname{rank}}
\newtheorem{theorem}{Theorem}[section]
\newtheorem{lemma}[theorem]{Lemma}
\numberwithin{equation}{section}
\expandafter\let\expandafter\oldproof\csname\string\proof\endcsname
\let\oldendproof\endproof
\renewenvironment{proof}[1][\proofname]{%
  \oldproof[\bf #1]%
}{\oldendproof}
\def\Li{{L^1}}
\def\LI{{L^1(0,1)}}
\def\li{{\ell^1}}
\def\Linf{{C}}
\def\LInf{{C[0,1]}}
\def\linf{{c}}
\def\Id{\mathop{\mathrm{Id}}}
\def\vx{\mathbf{x}}
\def\vy{\mathbf{y}}
\def\ve{\mathbf{e}}

\begin{document}

\title{Strict s-numbers of the Volterra operator}

\begin{abstract}
For Volterra operator $V\colon \LI \to \LInf$
and summation operator $\sigma\colon \li \to \linf$,
we obtain exact values of Approximation, Gelfand,
Kolmogorov, Mityagin and Isomorphism numbers.
\end{abstract}

\author[\"O. Bak\c si]{\"{O}zlem Bak\c{s}i\textsuperscript{1,2}}
\author[T. Khan]{Taqseer Khan\textsuperscript{1,3}}
\author[J. Lang]{Jan Lang\textsuperscript{1}}
\author[V. Musil]{V\'\i t Musil\textsuperscript{1,4}}

\email[\"O.~Bak\c si]{baksi@yildiz.edu.tr}
\email[T.~Khan]{taqi.khan91@gmail.com}
\email[J.~Lang]{lang@math.osu.edu}
\urladdr{0000-0003-1582-7273}
\email[V.~Musil]{musil@karlin.mff.cuni.cz}
\urladdr{0000-0001-6083-227X}

\address{\textsuperscript{1}%
Department of Mathematics,
Ohio State University,
Columbus OH,
43210-1174 USA}

\address{\textsuperscript{2}%
Yildiz Technical University,
Faculty of Art and Science,
Department of Mathematics,
Istanbul,
Turkey}

\address{\textsuperscript{3}%
Aligarh Muslim University,
Aligarh, U.P.-202002,
India}

\address{\textsuperscript{4}%
Department of Mathematical Analysis,
Faculty of Mathematics and Physics, 
Charles University,
So\-ko\-lo\-vsk\'a~83,
186~75 Praha~8,
Czech Republic}  

\date{\today}

\subjclass[2010]{Primary 47B06, Secondary 47G10}
\keywords{%
Integral operator, summation operator, $s$-numbers.
}

\thanks{%
This research was partly supported
by the United States -- India Educational Foundation (USIEF)
and by the grant P201-13-14743S of the Grant Agency of the Czech Republic.
}

\maketitle

\bibliographystyle{alpha}

\section{introduction and main results}

Compact operators and their sub-classes (nuclear operators, Hilbert-Schmidt
operators, etc.)  play a crucial role in many different areas of Mathematics. 
These operators are studied extensively  but somehow less attention is devoted
to  operators which
are non-compact but close to the class of compact operators. In this work we
will focus on two such operators.

First, consider the Volterra operator~$V$,
given by
\begin{equation} \label{Vdef}
	Vf(t) = \int_0^t f(s)\,\d s,
	\quad (0\le t\le 1),
	\quad\text{for $f\in \LI$.}
\end{equation}
When $V$ is regarded as an operator
from $L^p$ into $L^q$, ($1<p,q<\infty$), it
is a compact operator, however
in the limiting case, when
$V$ maps $L^1$ into the space
$C$ of continuous functions on the closed unit interval,
the operator is
bounded, with the operator norm $\|V\|=1$, but non-compact.
It is worth mentioning that, despite being non-compact or even weakly non-compact, this
operator possesses some good properties as being strictly singular (follows from \cite{BG},
or see \cite{Lef17}).
This makes Volterra operator, in the above-mentioned limiting case, an
interesting example of a non-compact operator ``close'' to the class of compact
operators.
The focus of our paper will be on obtaining exact values of strict $s$-numbers
for this operator.

Volterra operator was already extensively studied. Let us briefly recall those results related to our work.
The first credit goes to V.~I.~Levin \citep{Levin},
who computed explicitly the norm of $V$ between two $L^p$ spaces ($1<p<\infty$)
and described the extremal function which is connected with the function $\sin_p$.
Later on, E.~Schmidt in \citep{Schmidt} extended this result for $V\colon L^p\to L^q$,
where $1<p,q<\infty$.
This  operator was also studied in the context of Approximation theory
\citep{PinkusConstrApprox,Kolm,Tih,TB,BS,PinkusBook}.
Later a weighted version of this operator was studied in connection with
Brownian motion \citep{LLBook}, Spectral theory \citep{EEBook,EEBook2}
and Approximation theory \citep{ELBook}.

Recently, sharp estimates for Bernstein-numbers of $V$ in the
limiting case were obtained in \citep[Theorem 2.2]{Lef17}, more specifically,
\begin{equation} \label{LefBer}
	b_n(V) = \frac{1}{2n-1}
		\quad\text{for $n\in\N$},
\end{equation}
and also the estimates for the essential norm can be found in recent
preprint \citep{Letal}.

In our paper, we will compute exact values of all the remaining strict $s$-numbers, i.e.,
Approximation, Gelfand, Kolmogorov, Mityagin and Isomorphism numbers
denoted by
$a_n$, $c_n$, $d_n$, $m_n$ and $i_n$ respectively
(for the exact definitions see Section~\ref{sec:back}).
Our main result reads as follows.

\begin{theorem} \label{thm:main}
Let $V\colon \LI \to \LInf$ be defined as in \eqref{Vdef}.
Then
\begin{equation} \label{acd}
	a_n(V) = c_n(V) = d_n(V) = \frac{1}{2}
	\quad\text{for $n\ge 2$}
\end{equation}
and
\begin{equation} \label{bmi}
	m_n(V) = i_n(V) = \frac{1}{2n-1}
	\quad\text{for $n\in\N$}.
\end{equation}
\end{theorem}

If we also include the result \eqref{LefBer} concerning the Bernstein numbers,
we see that all the strict $s$-numbers of $V$
split between two groups. The upper half \eqref{acd} remains bounded from
below while the lower half \eqref{bmi} converges to zero.
This phenomenon for this operator was already observed;
for instance compare \citep{BS} and \citep{BG},
and for weighted version see
\citep{ELJMAA} and \citep{ELMathNachr}.

The similar results continue to hold for the sequence spaces and for the
discrete analogue of~$V$, namely for the operator
$\sigma\colon \li \to \linf$,
defined as
\begin{equation} \label{sdef}
	\sigma(\vx)_k = \sum_{j=1}^k x_j,
		\quad (k\in\N),
		\quad\text{for $\vx \in\li$,}
\end{equation}
where we denoted $\vx = \{x_j\}_{j=1}^\infty$ for brevity.
The operator is well-defined and bounded with the operator norm $\|\sigma\|=1$.
It is shown in \citep[Theorem 3.2]{Lef17} that
\begin{equation*}
		b_n(\sigma) = \frac{1}{2n-1}
		\quad\text{for $n\in\N$.}
\end{equation*}
We have the next result.

\begin{theorem} \label{thm:discr}
Let $\sigma\colon \li\to \linf$ be the operator from \eqref{sdef}.
Then
\begin{equation*}
	a_n(\sigma) = c_n(\sigma) = d_n(\sigma) = \frac{1}{2}
	\quad\text{for $n\ge 2$}
\end{equation*}
and
\begin{equation*}
	m_n(\sigma) = i_n(\sigma) = \frac{1}{2n-1}
	\quad\text{for $n\in\N$}.
\end{equation*}
\end{theorem}

The proofs are provided at the end of Section~\ref{sec:proofs}.

\section{background material} \label{sec:back}

We shall fix the notation in this section, although we mostly work with standard notions from functional analysis.

\subsection{Normed linear spaces}

For normed linear spaces $X$ and $Y$, we denote by $B(X,Y)$ the set of
all bounded linear operators acting between $X$ and $Y$. For any $T\in B(X,Y)$,
we use just $\|T\|$ for its operator norm, since the domain and target spaces
are always clear from the context.
By $B_X$, we mean the closed unit ball of $X$ and,
similarly, $S_X$ stands for the unit sphere of $X$. It is well-known fact that $B_X$ is compact if and only if $X$ is finite-dimensional.

Let $Z$ be a closed subspace of the normed space $X$. The quotient space $X/Z$ is the
collection of the sets $[x] = x+Z = \{x+z;\,z\in Z\}$ equipped with the norm
\begin{equation*}
	\| [x] \|_{X/Z} = \inf\{ \|x-z\|_X;\,z\in Z\}.
\end{equation*}
We sometimes adopt the notation $\|x\|_{X/Z}$ when no confusion is likely to happen.
Recall the notion of canonical map $Q_Z\colon X\to X/Z$, given by $Q_Z(x)=[x]$.

By the Lebesgue space $L^1$, we mean the set of all real-valued, Lebesgue integrable
functions on $(0,1)$ identified almost everywhere and equipped with
the norm
\begin{equation*}
	\|f\|_1 = \int_{0}^{1} |f(s)|\,\d s.
\end{equation*}
The space of real-valued, continuous functions on $[0,1]$, denoted by $C$, enjoys the norm
\begin{equation*}
	\|f\|_\infty = \sup_{0\le t\le 1} |f(t)|.
\end{equation*}

The discrete counterpart to $L^1$ is
the space of all summable sequences, $\li$, where
\begin{equation*}
	\|\vx\|_1 = \sum_{j=1}^\infty |x_j|
\end{equation*}
and, similarly to the space $C$, we denote by $c$ the space of all convergent sequences
endowed with the norm
\begin{equation*}
	\|\vx\|_\infty = \sup_{j\in\N} |x_j|.
\end{equation*}
Here and in the latter, we use the abbreviation $\vx =
\{x_j\}_{j=1}^\infty$ for the sequences and we write them in bold font.
Note that we also consider only real-valued sequence spaces.

All the above-mentioned spaces are complete, i.e., they form Banach spaces.

\subsection{$s$-numbers}

Let $X$ and $Y$ be Banach spaces.
To every operator $T \in B(X, Y)$, one can attach a sequence of non-negative
numbers $s_n(T)$ satisfying for every $n\in\N$ the following conditions
\begin{enumerate}[(S1)]
	\item $\| T \| = s_1(T) \geq s_2 (T) \geq \dots \geq 0$,
	\item $s_n(T + S) \leq s_n (T) + \|S\|$ for every $S \in B(X, Y)$,
	\item $s_n(B\circ T\circ A) \leq \| B \| s_n(T) \| A \|$
		for every $A \in B(X_1, X)$ and $B \in B(Y,Y_1)$,
	\item $s_n( \Id\colon \ell_n^2 \rightarrow \ell_n^2) = 1$,
	\item $s_n(T) = 0$ whenever $\rank T< n$.
\end{enumerate}
The number $s_n(T)$ is then called the $n$-th $s$-number of the operator $T$.
When (S4) is replaced by a stronger condition
\begin{enumerate}[(S1)] \setcounter{enumi}{5}
	\item $s_n( \Id\colon E\to E) = 1$ for every Banach space $E$, $\dim E = n$,
\end{enumerate}
we say that $s_n(T)$ is the $n$-th strict $s$-number of $T$. 

Note that the original definition of $s$-numbers, which was introduced by Pietsch
in \citep{PietschStudia}, uses the condition (S6)
which was later modified to accommodate wider class of $s$-numbers (like Weyl, Chang and
Hilbert numbers).
For a detailed account of $s$-numbers, one is referred for instance to
\cite{PietschBook}, \cite{CSBook} or \citep{ELBook}.

We shall briefly recall some particular strict $s$-numbers.
Let $T \in B(X,Y)$ and $n \in\N$. Then the $n$-th Approximation, Gelfand, Kolmogorov, Isomorphism, Mityagin and Bernstein numbers of $T$ are defined by
\begin{align*}
a_n(T) & = \inf_{\substack{ F \in B(X,Y) \\ \rank F < n}}
\| T-F \|,
	\\
c_n(T) & = \inf_{\substack{ M \subseteq X \\ \codim M < n}}
	\sup_{x\in B_M}  \| Tx  \|_Y,
	\\
d_n(T) & = \inf_{\substack{N \subseteq Y \\ \dim N < n}}
	\sup_{x\in B_X} \| Tx \|_{Y/N},
	\\
i_n(T) & = \sup  \| A \|^{-1} \| B \|^{-1},
\intertext{where the supremum is taken over all Banach spaces $E$ with $\dim E \geq n$ and
$A \in B(Y, E)$, $B \in B(E, X)$ such that $A\circ T\circ B$ is the identity
map on $E$,}
m_n(T) & =
	\sup_{\rho\ge 0}
	\sup_{\substack{ N \subseteq Y \\ \codim N \ge n}}
	Q_N T B_X \supseteq \rho B_{Y/N},
	\\
\intertext{and}
	b_n(T) & = \sup_{\substack{M \subseteq X \\ \dim M \ge n}}
		\inf_{x \in S_M} \| Tx \|_Y,
\end{align*}
respectively.

\section{Proofs} \label{sec:proofs}

\begin{lemma} \label{lemm:isomorphism}
Let $n\in\N$.
Then we have the following lower bounds of the Isomorphism numbers of $V$ and $\sigma$.
\begin{enumerate}[\rm (i)]
\item 
\begin{equation*}
	i_n(V) \geq \frac{1}{2n-1};
\end{equation*}
\item 
\begin{equation*}
	i_n(\sigma) \geq \frac{1}{2n-1}.
\end{equation*}
\end{enumerate}
\end{lemma}

\begin{proof}
(i)
Let $n\in\N$ be fixed.
We shall construct a pair of maps $A$ and $B$ such that the chain
\def\lio{\ell^{1}_{w,n}}
\begin{equation*} 
	\lio
		\xrightarrow{B} \Li
		\xrightarrow{V} \Linf
		\xrightarrow{A} \lio
\end{equation*}
forms the identity on $\lio$. Here $\lio$ is the $n$-dimensional weighted space $\li$ with the norm given by
\begin{equation*}
	\|\vx\|_{\lio} = \sum_{k=1}^n w_k |x_k|.
\end{equation*}
For the purpose of this proof, we choose $w_k = 2$ for $1\le k \le n-1$ and $w_n=1$.

Now, define $A\colon\Linf\to \lio$ by
\begin{equation*}
	(Af)_k = (2n-1)\, f \biggl(\frac{2k-1}{2n-1}\biggr),
		\quad (1\le k\le n),
		\quad\text{for $f\in\Linf$}.
\end{equation*}
Obviously, $A$ is bounded with the operator norm
\begin{equation*}
	\|A\| = (2n-1)^2.
\end{equation*}

In order to construct the mapping $B$, consider the partition of the unit interval into subintervals $I_1$, $I_2$, \ldots, $I_{2n-1}$ of the same length, i.e.
\begin{equation*}
	I_k = \biggl[ \frac{k-1}{2n-1}, \frac{k}{2n-1} \biggr]
		\quad\text{ for $1\le k \le 2n-1$,}
\end{equation*}
and define
\begin{equation*}
	B(\vx) = \sum_{k=1}^{n-1} x_k \bigl( \chi_{I_{2k-1}} - \chi_{I_{2k}} \bigr)
		+ x_n \chi_{I_{2n-1}}.
\end{equation*}
Clearly $B(\vx)$ is integrable for every $\vx\in\lio$, $B$ is bounded and the operator norm
satisfies
\begin{equation*}
	\|B\| = \frac{1}{2n-1}.
\end{equation*}

One can observe that the composition $A\circ V\circ B$ is the identity mapping on $\lio$
and by the very definition of the $n$-th Isomorphism number we have
\begin{equation*}
	i_n(V) \ge \| A \|^{-1} \| B \|^{-1} = \frac{1}{2n-1}
\end{equation*}
which completes the proof.

As for the discrete case (ii), we consider the chain
\def\linfn{{\ell^\infty_n}}
\begin{equation*}
	\linfn
		\xrightarrow {B} \li
		\xrightarrow {\sigma} \linf
		\xrightarrow {A} \linfn
\end{equation*}
where $\linfn$ stands for the $n$-dimensional space $\ell^\infty$,
$A$ is given by
\begin{equation*}
	A(\vx)_k = x_{2k-1},
		\quad (1\le k\le n),
		\quad\text{for $\vx\in\linf$}
\end{equation*}
and $B$ satisfies
\begin{equation*}
	B(\vy) = (y_1, -y_1,\ldots, y_{n-1}, -y_{n-1}, y_n,0,0,\dots)
		\quad\text{for $\vy\in\linfn$.}
\end{equation*}
Both $A$ and $B$ are bounded, the composition $A\circ\sigma\circ B$ forms the identity on $\linfn$
and thus
\begin{equation*}
	i_n(\sigma) \ge \| A \|^{-1} \| B \|^{-1} = \frac{1}{2n-1},
\end{equation*}
since $\|A\| = 1$ and $\|B\| = 2n-1$.
\end{proof}

\begin{lemma} \label{lemm:Mityagin}
Let $n\in\N$ then we have the following estimates of the Mityagin numbers of $V$ and $\sigma$.
\begin{enumerate}[\rm (i)]
\item 
\begin{equation*}
	m_n(V) \le \frac{1}{2n-1};
\end{equation*}
\item 
\begin{equation*}
	m_n(\sigma) \le \frac{1}{2n-1}.
\end{equation*}
\end{enumerate}
\end{lemma}

\begin{proof}
(i)
Fix some $n\in\N$ and any $0<\varrho<m_n(V)$. By the definition of the $n$-th Mityagin number,
there is a subspace of $\Linf$, say $N$, such that $\codim N \ge n$ and
\begin{equation} \label{eq:rhoball}
	Q_N V B_{\Li} \supseteq \varrho B_{\Linf/N}.
\end{equation}
Define
\begin{equation*}
	E = \bigl\{ f\in \Li;\, \spn\{Vf\}\cap N = \{0\} \bigr\}
\end{equation*}
and observe that, since $V$ is injective, $E$ is a subspace of dimension at most $n$
and satisfies
\begin{equation} \label{eq:rhoball2}
	Q_N V B_{\Li} = \bigl\{ [Vf];\, f\in B_E \bigr\}.
\end{equation}
Thus, by \eqref{eq:rhoball} and \eqref{eq:rhoball2}, we have
\begin{equation*}
	\|Vf\|_{\infty} \ge \|Vf\|_{\Linf/N} \ge \varrho
	\quad\text{for $f\in S_E$,}
\end{equation*}
and hence
\begin{equation*}
	\|Vf\|_{\infty} \ge \varrho \|f\|_1
	\quad\text{for $f\in E$.}
\end{equation*}
Therefore, thanks to \citep[Lemma~2.4]{Lef17},
\begin{equation} \label{eq:Lefineq}
	\varrho \le \frac{1}{2n-1}
\end{equation}
and the lemma follows by taking the limit $\varrho\to m_n(V)$.

(ii) The exact analogy holds in the discrete case, as
for every $0<\varrho<m_n(\sigma)$ one can find a $n$-th dimensional
subspace $E$ in $\ell^1$, such that
\begin{equation*}
	\|\sigma(\vx)\|_\infty \ge \varrho \|\vx\|_1
	\quad\text{for $\vx\in E$}.
\end{equation*}
The assertion \eqref{eq:Lefineq} then also holds by \citep[Lemma~2.4]{Lef17}
and its discrete modification in the proof of \citep[Theorem~3.2]{Lef17}.
\end{proof}

\begin{lemma} \label{lemm:Kolmogorov}
Let $n\ge 2$ then the following lower bounds of the Kolmogorov numbers of $V$ and $\sigma$ hold.
\begin{enumerate}[\rm (i)]
\item 
\begin{equation*}
	d_n(V) \ge \frac{1}{2};
\end{equation*}
\item 
\begin{equation*}
	d_n(\sigma) \ge \frac{1}{2}.
\end{equation*}
\end{enumerate}
\end{lemma}

\begin{proof}
(i) Fix arbitrary $n\ge 2$ and $\varepsilon >0$. By the very definition of
the $n$-th Kolmogorov number, there exists a subspace of $\Linf$, say $N$,
such that $\dim N < n$ and
\begin{equation} \label{eq:koldef}
	d_n(V)+\varepsilon
		\ge \sup_{f\in B_{\Li}} \|Vf\|_{\Linf/N}.
\end{equation}
Let us define the trial functions
\begin{equation} \label{eq:stairs}
	f_k = 2^{k+1} \chi_{(2^{-k-1},2^{-k})},
		\quad (k\in\N).
\end{equation}
There is $\|f_k\|_1 = 1$ for every $k\in\N$.
Now, by the definition of the quotient norm, to every $Vf_k$ one can attach a
function $g_k \in N$
in a way that
\begin{equation} \label{eq:soclose}
	\|Vf_k - g_k\|_\infty \le \|Vf_k\|_{\Linf/N} + \varepsilon.
\end{equation}
Observe that the set of all the functions $g_k$ is bounded in $N$. Indeed, by \eqref{eq:koldef} and \eqref{eq:soclose},
\begin{align*}
	\|g_k\|_\infty
		\le \|Vf_k - g_k\|_\infty + \|Vf_k\|_\infty
		\le \|Vf_k\|_{\Linf/N} + \varepsilon + \|V\|\|f_k\|_1
		\le d_n(V) + 2\varepsilon + 1
\end{align*}
for every $k\in\N$. Thus, since $N$ is finite-dimensional, there is a
convergent subsequence of $\{g_k\}$ which we denote $\{g_k\}$ again. Hence
$g_k$ converges to, say, $g\in N$, i.e.~there is an index $k_0$ such
that $\|g_k - g\|_\infty < \varepsilon$ for every $k\ge k_0$.
The limiting function $g$ then satisfies
\begin{equation*}
	\|Vf_k - g\|_\infty \le \|Vf_k - g_k\|_\infty + \|g_k - g\|_\infty
		\le \|Vf_k\|_{\Linf/N} + 2\varepsilon
\end{equation*}
for $k\ge k_0$ and thus
\begin{equation} \label{eq:dnsup}
	\sup_{k\ge k_0} \|Vf_k - g \|_\infty
		\le d_n(V) + 3\varepsilon.
\end{equation}
Next, we estimate the left hand side of \eqref{eq:dnsup} by taking the value attained at zero,
i.e.
\begin{equation} \label{eq:eval1}
	\sup_{k\ge k_0} \|Vf_k - g \|_\infty
	\ge \sup_{k\ge k_0} |Vf_k(0) - g(0) |
	= |g(0)|,
\end{equation}
or at the points $2^{-k}$, i.e.
\begin{equation} \label{eq:eval2}
	\sup_{k\ge k_0} \|Vf_k - g \|_\infty
	\ge \sup_{k\ge k_0} |Vf_k(2^{-k}) - g(2^{-k}) |
	\ge |1-g(0)|,
\end{equation}
where we used that $g$ is continuous.
Combining \eqref{eq:dnsup}, \eqref{eq:eval1} and \eqref{eq:eval2}, we get
\begin{equation*}
	d_n(V) + 3\varepsilon
	\ge \max \bigl\{ |g(0)|, 1-|g(0)| \bigr\}
	\ge \frac{1}{2}.
\end{equation*}
Therefore $d_n(V) \ge 1/2$ by the arbitrariness of $\varepsilon$.

(ii) The proof needs just slight modifications here. Instead of functions $f_k$,
one can use the canonical vectors $\ve^k=(0,\ldots,0,1,0,\ldots)$, where
the element $1$ is placed at the $k$-th coordinate. If we follow the previous lines,
we find close points $\vy^k\in N$, their limit~$\vy$ in~$\linf$ and we end up with
\begin{equation} \label{eq:dnsups}
	\sup_{k\ge k_0} \| \sigma(\ve^k) - \vy \|_\infty
		\le d_n(\sigma) + 3\varepsilon.
\end{equation}
Now, we estimate the supremum in \eqref{eq:dnsups} by
taking the $k$-th coordinate, i.e.
\begin{equation*} 
	\sup_{k\ge k_0} \| \sigma(\ve^k) - \vy \|_\infty
	\ge \sup_{k\ge k_0} |\sigma(\ve^k)_k - y_k |
	\ge |1-\lim_{k\to\infty} y_k|,
\end{equation*}
or by taking the coordinate $k-1$, i.e.
\begin{equation*} 
	\sup_{k\ge k_0} \| \sigma(\ve^k) - \vy \|_\infty
	\ge \sup_{k\ge k_0} |\sigma(\ve^k)_{k-1} - y_{k-1} |
	\ge |\lim_{k\to\infty} y_k|,
\end{equation*}
and the conclusion follows.
\end{proof}

\begin{lemma} \label{lemm:Gelfand}
Let $n\ge 2$ then the estimates of Gelfand numbers of $V$ and $\sigma$ read as
\begin{enumerate}[\rm (i)]
\item 
\begin{equation*}
	c_n(V) \ge \frac{1}{2};
\end{equation*}
\item 
\begin{equation*}
	c_n(\sigma) \ge \frac{1}{2}.
\end{equation*}
\end{enumerate}
\end{lemma}

\begin{proof}
(i)
Let $n\ge 2$ and $\varepsilon>0$ be fixed. By the definition of the $n$-th
Gelfand number, we can find a subspace $M$ in $\Li$ having $\codim M < n$ and satisfying
\begin{equation} \label{eq:Gelfdef}
	c_n(V) + \varepsilon
		\ge \sup_{f\in B_{M}} \|Vf\|_\infty.
\end{equation}
The proof will be finished once we show that the supremum in \eqref{eq:Gelfdef} is
at least one half.
We make use the step functions $f_k$, ($k\in\N$), defined by \eqref{eq:stairs} in the proof
of Lemma~\ref{lemm:Kolmogorov}. Recall that $\|f_k\|_1=1$, $\|f_k-f_l\|_1 = 2$
and $\|Vf_k - Vf_l\|_{\infty} = 1$ for $k\ne l$.
Note that the quotient space $\Li/M$ is of finite dimension thus,
the projected sequence $\{[f_k]\}$ is bounded and hence there is a Cauchy
subsequence, which we denote $\{[f_k]\}$ again.
Now, let $\eta>0$ be fixed. We have
\begin{equation} \label{eq:fifj}
	\|f_k - f_l\|_{\Li/M} < \eta
\end{equation}
for $k$ and $l$ sufficiently large. Let us denote $f = \frac{1}{2} (f_k - f_l)$ for these $k$ and $l$. Thanks to \eqref{eq:fifj} and the definition of quotient norm, one can find
a function $g\in M$ such that
\begin{equation*}
	\|f -g\|_1 \le \eta.
\end{equation*}
On setting
\begin{equation*}
	h = \frac{g}{1+\eta},
\end{equation*}
we have
\begin{equation*}
	\|h\|_1 \le \frac{1}{1+\eta}
		\bigl( \|f\|_1 + \|f-g\|_1 \bigr)
		\le 1,
\end{equation*}
whence $h\in B_M$.
Next
\begin{align*}
	\|Vh\|_\infty
		& \ge \frac{1}{1+\eta}
			\bigl( \|Vf\|_\infty - \|V(f-g)\|_\infty \bigr)
			\\
		& \ge \frac{1}{1+\eta}
			\biggl( \frac{1}{2} - \|V\| \|f-g\|_1 \biggr)
			\\
		& \ge \frac{1}{1+\eta}
			\biggl( \frac{1}{2} - \eta \biggr),
\end{align*}
thus
\begin{equation*}
	\sup_{f\in B_{M}} \|Vf\|_\infty
		\ge \frac{1}{1+\eta}
			\biggl( \frac{1}{2} - \eta \biggr)
\end{equation*}
and the lemma follows, since $\eta>0$ was arbitrarily chosen.

The proof of the discrete counterpart (ii) is completely analogous, once we
consider the canonical vectors $\ve^k$ instead of $f_k$, hence we omit it.
\end{proof}

\begin{lemma} \label{lemm:approx}
Let $n\ge 2$.
Then we have the following upper bounds of the Approximation numbers of $V$ and $\sigma$.
\begin{enumerate}[\rm (i)]
\item 
\begin{equation*}
	a_n(V) \le \frac{1}{2};
\end{equation*}
\item 
\begin{equation*}
	a_n(\sigma) \le \frac{1}{2}.
\end{equation*}
\end{enumerate}
\end{lemma}

\begin{proof}
(i) Consider the one-dimensional operator $F\colon\Li\to\Linf$ given by
\begin{equation*}
	Ff(t) = \frac{1}{2}\int_0^1 f(s)\,\d s,
	\quad (0\le t \le1),
	\quad\text{for $f\in\Li$.}
\end{equation*}
Then $F$ is a sufficient approximation of $V$. Indeed,
\begin{align*}
	\|Vf - Ff\|_{\infty}
		& = \sup_{0\le t\le 1} \biggl| \int_0^t f(s)\,\d s - \frac{1}{2} \int_{0}^{1} f(s)\,\d s \biggr|
			\\
		& = \sup_{0\le t\le 1} \biggl| \frac{1}{2} \int_0^t f(s)\,\d s - \frac{1}{2} \int_{t}^{1} f(s)\,\d s \biggr|
			\\
		& \le \sup_{0\le t\le 1} \frac{1}{2} \int_{0}^{t} |f(s)|\,\d s
			+ \frac{1}{2} \int_{t}^{1} |f(s)|\,\d s
			\\
		& = \frac{1}{2} \|f\|_{1}
\end{align*}
and therefore
\begin{equation*}
	a_n(V)\le \|V-F\| \le \frac{1}{2}.
\end{equation*}

In order to show (ii), choose the operator
\begin{equation*}
	\varrho(\vx)_k = \frac{1}{2}\sum_{j=1}^\infty x_j,
		\quad (k\in\N),
		\quad \text{for $\vx\in\li$.}
\end{equation*}
This is a well-defined one-dimensional operator and, by the calculations similar to above,
$\|\sigma-\varrho\|\le 1/2$. The proof is complete.
\end{proof}

Now, we are at the position to prove the main results.

\begin{proof}[Proof of Theorem~\ref{thm:main}]
Let $n\ge 2$ be fixed. Since $a_n(V)$ is the largest between all $s$-numbers,
we immediately obtain the inequalities $a_n(V)\ge c_n(V)$ and $a_n(V)\ge d_n(V)$.
Using Lemma~\ref{lemm:approx} with Lemma~\ref{lemm:Gelfand} and Lemma~\ref{lemm:Kolmogorov}
we obtain
\begin{equation*}
	\textstyle
	\frac{1}{2}\ge a_n(V) \ge c_n(V) \ge \frac{1}{2}
\qquad\text{and}\qquad
	\frac{1}{2}\ge a_n(V) \ge d_n(V) \ge \frac{1}{2}
\end{equation*}
respectively. This gives~\eqref{acd}.

Next, let $n$ be arbitrary. Due to $i_n(V)$ being the smallest strict $s$-number,
we have that $i_n(V)\le b_n(V)$ and also $i_n(V)\le m_n(V)$. For the lower bound,
we use Lemma~\ref{lemm:isomorphism}, while for the upper, we make use of Lemma~\ref{lemm:Mityagin} and the result of Lef\`evre, \citep{Lef17}.
This gives \eqref{bmi}.
\end{proof}

\begin{proof}[Proof of Theorem~\ref{thm:discr}]
The proof follows along exactly the same lines as that of Theorem~\ref{thm:main}
and hence omitted.
\end{proof}

\end{document}